\documentclass{article} 
\usepackage{amssymb,amsmath,amsthm}
\usepackage{amscd}
\usepackage{hyperref}
\usepackage{epsfig}
\usepackage{amsfonts}
\usepackage{amssymb}
\usepackage{amsmath,enumerate}
\usepackage{commath}
\usepackage{euscript}
\usepackage{enumitem}
\usepackage{amscd}
\usepackage{xcolor}
\usepackage[ruled,vlined]{algorithm2e}
\usepackage[numbers,sort&compress]{natbib}

\newtheorem{theorem}{Theorem}[section]
\newtheorem{remark}{Remark}[section]
\newtheorem{corollary}{Corollary}[section]

\newtheorem{example}{Example}[section]

\makeatletter
\newsavebox{\@brx}
\newcommand{\llangle}[1][]{\savebox{\@brx}{\(\m@th{#1\langle}\)}%
	\mathopen{\copy\@brx\kern-0.5\wd\@brx\usebox{\@brx}}}
\newcommand{\rrangle}[1][]{\savebox{\@brx}{\(\m@th{#1\rangle}\)}%
	\mathclose{\copy\@brx\kern-0.5\wd\@brx\usebox{\@brx}}}
\makeatother
\begin{document}
	\title{Common fixed points for set-valued contraction on a metric space with graph}
	\author{
		Pallab Maiti\footnotemark[2], Asrifa Sultana\footnotemark[1] \footnotemark[2]}
	\date{ }
	\maketitle
	\def\thefootnote{\fnsymbol{footnote}}
	
	\footnotetext[1]{ Corresponding author. e-mail- {\tt asrifa@iitbhilai.ac.in}}
	\noindent
	\footnotetext[2]{Department of Mathematics, Indian Institute of Technology Bhilai, Raipur - 492015, India.
	}
	
	
		
	\begin{abstract}
	In this article, we derive a common fixed point result for a pair of single valued and set-valued mappings on a metric space having graphical structure. In this case, the set-valued map is assumed to be closed valued instead of closed and bounded valued.
	Several results regarding common fixed points and fixed points follow from the main theorem of this article. By applying our theorem, we deduce the convergence of the iterates for a nonlinear $q$-analogue Bernstein operator. Furthermore, we establish sufficient criteria for the occurrence of a solution to a fractional differential equation.
	\end{abstract}
	
\section{Introduction}
Suppose that $W$ is a non-void set endowed with a metric $d$. The set $CB(W)$ represents the collection of all non-void bounded closed subsets of $W$. A map $H:CB(W)\times CB(W) \to \mathbb{R}$ is described as a Hausdorff metric on $CB(W)$, whenever for $Y$ and $Z$ lies in $CB(W)$, 
\begin{center}
	$\displaystyle{	H(Y,Z)= \max\left\{ \sup_{u\in Y}  D(u,Z),\,\,\sup_{v\in Z}D(v,Y)\right\}}$,
\end{center}
whereas the notation $D(u,Z)$ describes the $\inf_{v\in Z}d(u,v)$. Further, $CL(W)$ expresses the collection of all non-void subsets of $W$, which are closed. Given a set-valued map $F:W\to CL(W)$, an element $w^*\in W$ is stated as a fixed point for $F$, if it lie in $\{w\in W: w\in Fw\}$. The occurrence of fixed points for a set-valued contraction first derived by Nadler \cite{nadler} in the year of 1969.
Later, Mizoguchi and Takahashi \cite{mizo} improved the Nadler's \cite{nadler} result for a set-valued non-linear contraction in the year of 1989. The authors \cite{mizo} deduced that if $F:W\to CB(W)$ and for each $v,w$ in $W$ fulfils $H(Fv, Fw)\leq k(d(v,w))d(v,w)$, where $k \in U=\{h:[0,\infty)\to [0,1)| \limsup_{r\to t^+}h(r)<1,\forall \,0\leq t<\infty\}$, then occurrence of a fixed point for the map $F$ is guaranteed whenever $(W,d)$ is complete. There are several generalization of this result in the literature, which can be found in \cite{berinde,mam, mam_feng}. In the year 2006, Feng and Liu \cite{feng} extended the well known Nadler's \cite{nadler} result for a set-valued generalized contraction $F$ from $W$ into $CL(W)$, where the Hausdorff distance function is not involved.
\begin{theorem}\cite{feng}\label{feng}
	For the metric space $(W,d)$, a set-valued map $F:W\to CL(W)$ in order that for each $v\in W$, $w\in Fv$ fulfils
	\begin{equation}
		D(w,Fw)\leq \theta d(v,w)~\textrm{where}~0\leq \theta<1 .
	\end{equation}
	Then there is an element $w^*\in Fw^*$ if the map $h: W\to \mathbb{R}$ given by $h(v)=D(v,Fv)$ is lower semicontinuous and $(W,d)$ is complete.       
\end{theorem}
For some more results about the fixed points for set-valued generalized contraction can be found in \cite{azam,klim}.

In the parallel development, many mathematicians scrutinized common fixed point results for a couple of single valued mappings on a metric space, one may see \cite{abbas, Doric}, whereas in \cite{khan,kumam}, the authors studied the common fixed point theory for a couple of single valued and set-valued mappings.
In the year 2007, Kamran \cite{kamran} established a common fixed point result for a single valued and set-valued generalized contractive mappings. This result actually enables the author \cite{kamran} to generalize several fixed point and common fixed point results on metric space available in the literature. Indeed, the author \cite{kamran} established the below stated result. 
\begin{theorem}\cite{kamran}\label{thm1}
	For the metric space $(W,d)$, the maps $F:W\to CB(W)$, $f:W\to W$ in order that for each $u\in W$, $Fu\subset fW$ and 
	\begin{equation}\label{eqn1}
		H(Fv,Fw)\leq k(d(fv,fw))d(fv,fw)+M\,D(fv,Fw)~\text{for each}~ v,w \in W,
	\end{equation}
	where $k\in U$, $M\geq 0$. Then occurrence of $w^*$ having $fw^*\in Fw^*$ is guaranteed if $fW$ is complete. Moreover, $w^*=fw^*\in Fw^*$, whenever $ffw^*=fw^*$ and $\{ffw^*\}\subset Ffw^*$.
\end{theorem}
\noindent We note that, the above mentioned result deduces Mizoguchi-Takahashi's \cite{mizo} theorem by considering $f$ to be an identity map and $M=0$.


On the other side, Jachymski \cite{Jach} extended the renowned Banach contraction principle via $G$-contraction on a metric space having  graph $G$. With the aid of this result, the author \cite{Jach} enables to extend and unify some fixed point results proved in metric spaces as well as metric spaces having partial order.  
Followed by Jachymski \cite{Jach}, an extension of the Theorem \ref{feng} due to Feng and Liu \cite{feng} was derived by Sultana and Vetrivel \cite{mam_feng}  in the year 2017, on a metric space consisting with directed graph $G$. 
The authors \cite{mam_feng} demonstrated the below stated result.  
\begin{theorem}\cite{mam_feng}\label{mam}
	Let $(W,d)$ along with a directed graph $G$ and a map $F:W\to CL(W)$ in order that for every $v \in W$, $w\in Fv$ with $(v,w)\in E(G)$ fulfils
	\begin{enumerate}
		\item [(i)] $D(w,Fw)\leq k(d(v,w))d(v,w)$ for $k\in U$,
		\item [(ii)] for any $p\in Fw$ having $d(p,w)\leq d(v,w)$ indicates $(p,w)\in E(G)$. 
	\end{enumerate}
	Suppose that $w_0\in W$,  $w_1\in Fw_0$ with $(w_0,w_1)\in E(G)$. Moreover, assume that if any sequence $\{w_n\}_n\subset W$ with $(w_{n-1},w_{n})\in E(G)$ for every $n$, and $w_n\to w$ then there is $\{w_{n_r}\}_{r\in \mathbb{N}}$ in order that $D(w,Fw)\leq \liminf_{r\to \infty}D(w_{n_r},F(w_{n_r}))$. Then there is an element $w^*\in Fw^*$ if $(W,d)$ is complete.
\end{theorem}
\noindent As an application of the Theorem \ref{mam}, the authors \cite{mam_feng} improved the Kelisky and Rivlin \cite[Theorem 1]{KR8} result about the convergence of iterates for Bernstein operator on a complete normed linear space. In the literature, there are several fixed point theorems using the notion of graph theory, which can be found in \cite{frigon,khan, mam}.

Inspired by Kamran \cite{kamran} and Sultana-Vetrivel \cite{mam_feng}, 
in this present manuscript, we derive a common fixed point result on a metric space having graphical structure for a set-valued mapping, which is supposed to be closed valued instead of bounded and closed valued. 
The main result is indeed an extension of the Theorem \ref{thm1} due to Kamran \cite{kamran} and the Theorem \ref{mam} due to Sultana-Vetrivel \cite{mam_feng}. 
As an implementation of our result,  we establish the convergence of iterates for some nonlinear operators, which enables us to derive the convergence of iterates for a nonlinear $q$-analogue Bernstein operator. As an another application, we derive a sufficient criteria for the presence of a solution for a fractional order differential equation. An invariant best approximation result is also discussed in this article.
\bigskip
\section{Preliminaries}
In this part, we utilize some essential symbols and definitions, which is needful all throughout this article.
An element $w^*$ is described as a coincidence point for a couple of mappings $F:W\to CL(W)$ and $f:W\to W$, if $\{fw^*\}\subset Fw^*$. Furthermore, an element $w^*$ is described as a common fixed point for the map $f$ coupled with $F$ if $w^*=fw^*\in Fw^*$. Additionally, the map $f$ is stated as a $F$-weakly commuting \cite{kamran} at an element $z^*$ in $W$ if $ffz^*\in Ffz^*$. Up to the end, the collection of common fixed points and coincidence points for the maps $f$ and $F$ are expressed by $Fix(f,F)$ and $Coin(f,F)$ respectively.

Now we present some fundamental terminology of graph theory (see \cite{John}), which is mandatory up to the end of this article. For the metric space $(W,d)$, let us assume a directed graph $G$ which is a pair of $(V(G),E(G))$, whereas the collection of all vertices are symbolized by $V(G)$, which is actually whole set $W$ and the collection of edges $E(G)$ includes the set $\Delta=\{(s,s):(s,s)\in W\times W\}$. Further, this reflexive graph $G$ does not possess any parallel edges.

\section{Main results}
Up to the end of this part, we consider a metric space $(W,d)$ along with a directed graph $G(V(G),E(G))$, whereas $V(G)=W$, $\Delta \subseteq E(G)$ and $G$ has no parallel edges. In the below stated theorem, we scrutinize common fixed points on a metric space having a graph $G$ for a set-valued generalized contraction, which is assumed to be only closed-valued. 
\begin{theorem}\label{thm3.2}
	For the metric space $(W,d)$, the maps $F:W \to CL(W)$, $f:W\to W$ in order that for each $v\in W$, $fw\in Fv$ with $(fv,fw)\in E(G)$ fulfils,
	\begin{enumerate}\label{main}
		\item[(i)] $D(fw,Fw)\leq k(d(fv,fw))d(fv,fw)$, where $k\in U$,
		\item[(ii)] for any $fp\in Fw$ with $d(fw,fp)\leq d(fv,fw)$, indicates $(fw,fp)\in E(G)$.
	\end{enumerate} 
	Suppose that $w_0\in W$, $p_0\in F(w_0)$ with $(fw_0,p_0)\in E(G)$. Also assume that if any $\{fu_n\}_{n}\subset W$ with $(fu_{n-1},fu_{n})\in E(G)$, $fu_n\in Fu_{n-1}$ for every $n\in \mathbb{N}$ and $fu_n\to fu$, then there is $\{fu_{n_r}\}\subseteq\{fu_n\}$ in order that $D(fu,Fu)\leq \liminf_{r\to \infty}D(fu_{n_r},Fu_{n_r})$. Then there is a sequence $\{fw_n\}_{n}$ with $fw_n\in Fw_{n-1}$ converges to $fw^*$, where $fw^*\in Fw^*$ 
	is guaranteed if for each $u\in W$, $Fu\subset fW$ and $(fW,d)$ is complete.
	
	Moreover, $f$ and $F$ possess a common fixed point whenever the map $f$ is weakly commuting at $w^*$ along with $fw^*=ffw^*$.

\end{theorem}


%

%
	
	\begin{proof}
		As $p_0\in Fw_0$ and $Fw_0\subset fW$, then there is $w_1\in W$ in order that $p_0=fw_1$. Consequently, we obtain $(fw_0,fw_1)\in E(G)$ and $fw_1\in F(w_0)$.
		Since $D(fw_1,Fw_1)=\inf\{d(fw_1,y): y\in Fw_1\}$, then for $\frac{1}{\sqrt{k(d(fw_0,fw_1))}}>1$ there is $w_2\in W$ with $fw_2\in F(w_1)$ fulfils
		\begin{eqnarray}\label{eqn_1st}
			d(fw_1,fw_2)&\leq& D(fw_1,Fw_1)+ \left(\frac{1}{\sqrt{k(d(fw_0,fw_1))}}-1\right)D(fw_1,Fw_1)\nonumber\\&\leq&\frac{1}{\sqrt{k(d(fw_0,fw_1))}} k(d(fw_0,fw_1))d(fw_0,fw_1)\nonumber\\
			&\leq&\sqrt{k(d(fw_0,fw_1))}d(fw_0,fw_1)< d(fw_0,fw_1).
		\end{eqnarray}
		Due to the fact that $F$ meets the hypothesis $(ii)$ and $(fw_0,fw_1)\in E(G)$ with $fw_1\in Fw_0$, then the previous inequation $(\ref{eqn_1st})$ leads to $(fw_1,fw_2)\in E(G)$.  
		Repeating the above methodology, for $\frac{1}{\sqrt{k(d(fw_1,fw_2))}}>1$ we get $fw_3\in F(w_2)$ fulfilling
		\begin{eqnarray*}
			d(fw_2,fw_3)&\leq& D(fw_2,Fw_2)+\left( \frac{1}{\sqrt{k(d(fw_1,fw_2))}}-1\right)D(fw_2,Fw_2)\nonumber\\
			&\leq& \frac{1}{\sqrt{k(d(fw_1,fw_2))}}k(d(fw_1,fw_2))d(fw_1,fw_2)< d(fw_1,fw_2)
		\end{eqnarray*}
		Hence by the same argument as before we have $(fw_2,fw_3)\in E(G)$. Following this approach, we are able to construct two sequences $\{f(w_n)\}_{n}$, $\{F(w_n)\}_{n}$ in order that for every $n$, $(fw_{n-1},fw_n)\in E(G)$ and $f(w_n)\in F(w_{n-1})$ fulfils
		\begin{equation}\label{eqn2}
			d(fw_n,fw_{n+1})
			\leq\sqrt{k(d(fw_{n-1},fw_n))}d(fw_{n-1},fw_n)<d(fw_{n-1},fw_n).
		\end{equation}
		Now for $n\geq 1$, we consider $s_n=d(fw_n,fw_{n+1})$. Consequently from $(\ref{eqn2})$, it is clearly observed that for each $n$, $s_n$ are non-negative and $s_{n+1}<s_n$.
		Therefore $\{s_n\}_n$ is convergent and assume $s_n \to q$, where $q\geq 0$. On the account of $\limsup_{t\to q^+}k(t)<1$, there is $0\leq \alpha<1$ and $M\geq 1$ meeting the condition $k(s_n)\leq \alpha$ for each $n\geq M$. Now for every $n> M$, the inequation $(\ref{eqn2})$ leads to
		\begin{eqnarray}\label{bounded}
			d(fw_n,fw_{n+1})&\leq&\alpha^{(n-M)/2}\prod_{j=1}^{M-1}\sqrt{k(d(fw_{j-1},fw_{j}))}d(fw_{0},fw_1)\nonumber\\
			&<&B\alpha^{n/2}d(fw_{0},fw_1),
		\end{eqnarray}
		where $B=\alpha^{-{M/2}}\prod_{j=1}^{M-1}\sqrt{k(d(fw_{j-1},fw_{j}))}$.
		Now using $(\ref{bounded})$ and for any arbitrary positive integer $m$ and $n>M$,
		\begin{equation*}
			\begin{aligned}
				d(fw_n,fw_{n+m})&\leq d(fw_{n},fw_{n+1})+d(fw_{n+1},fw_{n+2})+\cdots+\\
				& \qquad d(fw_{n+m-1},fw_{n+m})\\
				&\leq B\alpha^{n/2}\left[1+\alpha^{1/2}+ \cdots+ \alpha^{(m-1)/2}\right]d(fw_{0},fw_1)\\
				&\leq B\frac{\alpha^{n/2}}{1-\alpha^{1/2}}d(fw_{0},fw_1).
			\end{aligned}
		\end{equation*}
		Thus the sequence $\{fw_n\}_{n}\subset fW$ turns out to be Cauchy, on the account of $\alpha^{n/2}\to 0$ as $n\to \infty$. 
		Since $fW$ is complete, then $\{fw_n\}_{n}$ is converges to $fw^*$, for some $w^*\in W$.
		Again for every $n$, $(fw_{n-1},fw_n)\in E(G)$ and $fw_n\in Fw_{n-1}$, it yields from $(i)$ that $D(fw_n,Fw_n)\leq k(d(fw_{n-1},fw_n))d(fw_{n-1},fw_n)<d(fw_{n-1},fw_n)$. Hence $\lim_{n \rightarrow \infty}D(fw_n,Fw_n)=0$.

		On the account of $(fw_{n-1},fw_n)\in E(G)$ with $fw_{n}\in Fw_{n-1}$ for each $n$ and $fw_n\to fw^*$, consequently by the hypothesis there is a subsequence $\{fw_{n_r}\}_{r\in \mathbb{N}}$ having  $D(fw^*,Fw^*)\leq \liminf_{r\to \infty}D(fw_{n_r},F(w_{n_r}))$. Thus we can conclude that $D(fw^*,Fw^*)=0$. Due to the fact that $Fw^*$ is closed, therefore $fw^*\in Fw^*$.
		
		Now let $a=fw^*$, then the hypothesis $fw^*=ffw^*$ indicates that $fa=a$. Again $f$ and $F$ are weakly commuting at $w^*$, hence $ffw^*\in Ffw^*$ leads to $a=fa\in Fa$. 
	\end{proof}
	\begin{remark}
		Take a graph $G$ on the set $W$ in order that $V(G)=W$ and $E(G)= W\times W$. Consequently, the aforementioned Theorem \ref{feng} due to Feng-Liu \cite{feng} is followed from our main Theorem \ref{thm3.2} by considering the map $f$ is an identity map on $W$ and for each $0\leq t<\infty$, $k(t)=\theta$, where $\theta\in [0,1)$.
	\end{remark}
	
	The succeeding corollary deals with common fixed points for set-valued generalized contraction in a metric space, which is indeed an extension of the Theorem \ref{feng} due to Feng-Liu \cite{feng}. Moreover, this result is followed from the aforementioned Theorem \ref{thm3.2} by choosing a graph $G$ on $W$ in order that $V(G)=W$ and $E(G)=W\times W$.  
	\begin{corollary}\label{cor}
		For the metric space $(W,d)$, the maps $F:W \to CL(W)$, $f:W\to W$ in order that for each $v\in W$, $fw\in Fv$ fulfils,
		\begin{equation*}
			D(fw,Fw)\leq k(d(fv,fw))d(fv,fw), ~\textrm{where}~ k\in U.
		\end{equation*}
		Suppose that the function $h: W\to \mathbb{R}$ by $h(v)=D(fv,Fv)$ is lower semi-continuous. Then occurrence of $w^*$ having $fw^*\in Fw^*$ is guaranteed if for each $u\in W$, $Fu\subset fW$ and $(fW,d)$ is complete.
		
		Moreover, $f$ and $F$ possess a common fixed point whenever the map $f$ is weakly commuting at $w^*$ along with $fw^*=ffw^*$.
	\end{corollary}
	
	
	The below-stated example illustrates the preceding Theorem \ref{thm3.2}.
	\begin{example}\label{example}
		Consider a set $W=\{0,1\}\cup\left\{\frac{1}{3^n}:n\in \mathbb{N}\right\}$ together with usual metric. Suppose that a single valued map $f:W\to W$ by
		\begin{center}
			$f(w)=
			\begin{cases}
				~\frac{1}{3^{n+1}} &\mbox{if}~~w=\frac{1}{3^n},~n\in \mathbb{N}\cup\{0\}\\
				~0~&\mbox{if}~~w=0.
			\end{cases}$
		\end{center}
		and set-valued map $F: W\to CL(W)$ is defined as
		\begin{center}
			$F(w)=
			\begin{cases}
				~\left\{0,\frac{1}{3}\right\} &\mbox{if}~~w=0,\\
				~\left\{\frac{1}{3},\frac{1}{3^{n+2}}\right\} &\mbox{if}~~w=\frac{1}{3^n},~\text{where}~ n\in \mathbb{N},\\
				~\left\{0\right\}~&\mbox{if}~~w=1,
			\end{cases}$
		\end{center} 
		Take a graph $G$ on $W$ in order that $V(G)=W$ and $E(G)=\{(v,w)\in W\times W: d(v,w)< \frac{1}{9}\}$. Then, $G$ does not have any parallel edges and $E(G)$ contains the set $\Delta$. Now for $v=\frac{1}{3^n}$ and $fw=\frac{1}{3^{n+2}}\in Fv$, we see that for each $n$, $(fv,fw)\in E(G)$. Further,
		\begin{equation*}
			D(fw,Fw)=D\left(\frac{1}{3^{n+2}},\left\{\frac{1}{3},\frac{1}{3^{n+3}}\right\}\right)=\frac{1}{3^{n+2}}-\frac{1}{3^{n+3}}\leq \frac{1}{3}d(fv,fw).
		\end{equation*} 
		Again, for $v=0$ and $fw=0\in Fv$, the pair $(fv,fw)\in E(G)$. Also, it is easy to check that $D(fw,Fw)\leq \frac{1}{3}d(fv,fw)$. Thus we conclude that for each $v\in W$, $fw\in Fv$ with $(fv,fw)\in E(G)$ fulfils $D(fw,Fw)\leq k(d(fv,fw))d(fv,fw)$, where $k(t)=\frac{1}{3}$ for each $t\in [0,\infty)$. Furthermore, if $fp\in Fw$ with $d(fw,fp)\leq d(fv,fw)$, then $(fw,fp)\in E(G)$.
		Now for choosing $w_0=\frac{1}{3}$, we observe that $\frac{1}{27}\in Fw_0$ and $(fw_0,\frac{1}{27})\in E(G)$. It is clear to observe that the map $h:W\to \mathbb{R}$ by $h(w)=D(fw,Fw)$ is lower semi-continuous. Further, $(fW,d)$ is complete and for each $u\in W$, $Fu\subset fW$. Thus all the criteria of the Theorem \ref{thm3.2} is fulfilled, which confirm the occurrence of coincidence point $w^*$. We visualize that $0$ is a coincidence point of $f$ and $F$, that is, $f0\in F0$. Again $ff0=f0$ and $ff0\in Ff0$. We note that $0$ is a common fixed point of $f$ and $F$.
		
		On the other side, if we choose $v=\frac{1}{3^n}$ for each $n\geq 1$ and $fw=\frac{1}{3}\in Fw$, then $d(fv,fw)=\frac{1}{3}-\frac{1}{3^{n+1}}$ and $D(fw,Fw)=\frac{1}{3}$. Therefore for such chosen points $D(fw,Fw)>d(fv,fw)$. This indicates that the Theorem \ref{thm3.2} is indeed an extension of the Corollary \ref{cor}.  
	\end{example}	
	\begin{remark}
		It is worth to observe that the preceding Theorem \ref{thm3.2} is an extension of the Theorem \ref{thm1} due to Kamran \cite{kamran}. Indeed, consider $W$ incorporates with a graph $G$ having $V(G)=W$ and $E(G)=W\times W$. Let $v\in W$, $fw\in Fv$ with $(fv,fw)\in E(G)$, then it follows
		\begin{eqnarray*}
			D(fw,Fw)\leq H(Fv,Fw)&\leq& k(d(fv,fw))d(fv,fw)+M\,D(fw,Fv)\\
			&=& k(d(fv,fw))d(fv,fw),~\textrm{where}~ k\in U.
		\end{eqnarray*}
		Moreover, for the chosen graph $G$, it always follows that $(fw,fp)\in E(G)$.
		Also we are able to find $w_0,p\in W$ in order that $(fw_0,p)\in E(G)$ and $p\in Fw_0$. Now assume a $\{fw_n\}_n\in W$ in order that $fw_n\in Fw_{n-1}$ for every $n$, and $fw_n\to fw$. Evidently for each $n$, $(fw_n,fw_{n+1})\in E(G)$. Consequently for each $n$ and $fz\in Fw_n$,
		\begin{equation*}
			\begin{aligned}
				D(fw,Fw)
				&\leq  d(fw,fw_n)+d(fw_n,fz)+H(Fw_n,Fw)\\
				&\leq d(fw,fw_n)+ D(fw_n,Fw_n)+ d(fw,fw_n)+ M\,D(fw,Fw_n)\\
				&\leq d(fw,fw_n)+ D(fw_n,Fw_n)+ d(fw,fw_n).~ \textrm{[as $fw_n\in F(w_{n-1})]$}
			\end{aligned} 
		\end{equation*}
		Now taking $n\to \infty$ it yields, $D(fw,Fw)\leq\liminf_{n\to \infty} D(fw_n,Fw_n)$. Hence all the criteria of the Theorem \ref{thm3.2} is fulfilled.  
	\end{remark}
	
	The below-mentioned example demonstrates that, the Theorem \ref{thm3.2} indeed an extension of Theorem \ref{thm1} due to Kamran \cite{kamran}.
	\begin{example}
		In the Example \ref{example}, choose two points $0$ and $1$ in $W$. Then we see that $H(F0,F1)=\frac{1}{3}$, $d(f0,f1)=\frac{1}{3}$ and $D(f1,F0)=0$. Therefore for all $k\in U$ and $M\geq 0$, $H(F0,F1)=d(f0,f1)> k(d(f0,f1))d(f0,f1)+M\,D(f1,F0)$. This implies that $F$ does not fulfils the equation (\ref{eqn1})  of the Theorem \ref{thm1} due to Kamran. 
	\end{example}

	\begin{remark}
		The Theorem \ref{mam} due to Sultana and Vetrivel \cite{mam_feng} follows
		from our Theorem \ref{thm3.2} by choosing $f$ to be an identity map.
	\end{remark}
	
	Let the set $W$ be endowed with a norm $\|.\|$ and $Q\subseteq W$ with $Q\neq \emptyset$. Now for $z\in W$, a set $B_Q(z)$ contains $y\in Q$ in order that $D(z,Q)=\|y-z\|$. Then the set $B_Q(z)$ is called as a best $Q$-approximates of $z$ over $Q$. Invariant best approximation problems can be solved through the perception of common fixed points for single valued mappings, as one may see in \cite{thagafi,pathak}. Also in \cite{kamran1,kumam}, the authors established invariant best approximation results for a couple of single valued and set-valued mappings. In the succeeding theorem, we investigate invariant best approximation for single valued as well as set-valued maps meeting the condition $(i)$ of the Theorem \ref{thm3.2} on a normed linear space without considering the Hausdorff distance function. 
	\begin{theorem}
		Suppose that $(W,\|.\|)$ is a normed linear space and $Q\subseteq W$ such that $Q\neq \emptyset$. Assume that the mappings $f:Q\to Q$ and $F:Q \to CL(Q)$ fulfils the following:
		\begin{enumerate}
			\item[(i)] for every $v\in B_Q(z)$ and $fw\in Fv$ fulfils $D(fw,Fw)\leq k(\|fv-fw\|)\|fv-fw\|$ where $k \in U$, 
			
			\item[(ii)] $f(B_Q(z))=B_Q(z)$ and $B_Q(z)$ is complete,
			
			\item[(iii)] $\sup_{u\in Fp}\|u-z\|\leq \|fp-z\|$ for each $p\in B_Q(z)$. 
		\end{enumerate}
		Then $w^*\in Coin(f,F)\cap B_Q(z)$ if for any $\{fw_n\}_n\in B_Q(z)$ with $fw_n\in Fw_{n-1}$ and $fw_n\to fw$ implies $D(fw,Fw)\leq \liminf_{n\to \infty}D(fw_n,Fw_n)$. Furthermore, $Fix(f,F)\cap B_Q(z)\neq \emptyset$ if $ffw^*=fw^*$ and $ffw^*\in Ffw^*$. 
	\end{theorem}
	\begin{proof}
		Consider a graph $G$ on $B_Q(z)$ in order that $V(G)=B_Q(z)$ and $E(G)=B_Q(z)\times B_Q(z)$. Now assume an element $p$ lies in $B_Q(z)$ and $w\in Fp$. Then we have $fp\in B_Q(z)$ for every $p\in B_Q(z)$, on the account of $f(B_Q(z))=B_Q(z)$.
		Therefore we get $\|fp-z\|=D(z,Q)$, by using the definition of $B_Q(z)$. Again the condition $(iii)$ yields that $$\|w-z\|\leq \sup_{u\in Fp}\|u-z\|\leq \|fp-z\|=D(z,Q).$$	
		Hence $w\in B_Q(z)$. This indicates that  for each $p \in B_Q(z)$,  $Fp\subseteq B_Q(z)$. As the set $Fp$ is closed for each $p\in Q$, then $Fp$ is also closed for any $p \in B_Q(z)$. Subsequently, $F|_{ B_Q(z)}$ is a set-valued mapping from $B_Q(z)$ into $CL( B_Q(z))$ and $f|_{ B_Q(z)}: B_Q(z)\to  B_Q(z)$. Then eventually 
		$$Fix(f|_{ B_Q(z)},F|_{ B_Q(z)})=Fix(f,F)\cap B_Q(z).$$
		Thus this theorem is followed by applying our main Theorem \ref{thm3.2} under the assumption $W=B_Q(z)$. 	
	\end{proof}

	Jachymski \cite[Theorem 4.1]{Jach}, established the convergence of iterates for some linear operator with the perception of fixed point theory on a normed linear space, which is indeed an extension of Kelisky-Rivlin \cite[Theorem 1]{KR8} theorem about the convergence of iterates for Bernstein operator \cite{KR8}. Later, Sultana and Vetrivel \cite[Theorem 5]{mam_feng} improved this result on a complete normed linear space. The coming theorem deals with the convergence of iterates for certain operators in a normed linear space, where the operators need not to be linear. Essentially, our theorem generalizes the Sultana and Vetrivel's \cite[Theorem 5]{mam_feng} result.   
	
	\begin{theorem}\label{iterates}
		For the complete normed linear space $(W,\|.\|)$, let $W_0\subset W$ be a closed subspace. A map $T:W\to W$ with $(I-T)(W)\subseteq W_0$ and for $w-Tw\in W_0$,
		\begin{equation}\label{convergence}
			\|Tw-T^2w\|\leq k(\|w-Tw\|)\|w-Tw\|\quad\text{where}~k\in U.
		\end{equation} 
		Then for every $w\in W$, $\lim_{n \rightarrow \infty}T^nw\in \{w^*\in W: w^*=Tw^*\}\cap (w+W_0)$ if $h:W\to \mathbb{R}$ by $h(u)=D(u,Tu)$ is lower semicontinuous.
	\end{theorem}
	\begin{proof}
		Take a graph $G$ on $W$ in order that $V(G)=W$ and $E(G)$ is the collection of $(v,w)\in W\times W$ having $v-w\in W_0$. Consider $f$ is an identity map on $W$ and a set valued map $F$ is chosen by $Fw=\{Tw\}$, for each $w\in W$. Let $v$ be any arbitrary point in $W$ and $w=fw\in Fv=\{Tv\}$ with $(fv,fw)=(v,w)\in E(G)$. Hence $(v,Tv)\in E(G)$, that is $v-Tv\in W_0$. Consequently by (\ref{convergence}), $\|fw-Fw\|=\|Tv-T(Tv )\|=\|Tv-T^2v\|\leq k(\|v-Tv\|)\|v-Tv\|$ where $k\in U$. On the account of $(I-T)(W)\subseteq W_0$, it yields that for $p=fp \in Fw=\{Tw\}$, we have $w-p=w-Tw\in W_0$, that is $(w,p)=(fw,fp)\in E(G)$. Hence condition (i) and (ii) of Theorem \ref{thm3.2} is fulfilled. Choose an element $w\in W$. Moreover, for $w=w_0\in W$, there is $p_0\in F(w_0)$ with $(fw_0,fp_0)\in E(G)$ due to the fact that $(I-T)(W)\subseteq W_0$.
		Thus for $w\in W$, applying the previous Theorem \ref{thm3.2} we obtain that $\lim_{n\to \infty}T^nw\in \{w^*\in W: w^*=Tw^*\}$.
		
		As $(I-T)W\subseteq W_0$, then it is easy observe that $w-T^nw\in W_0$ for each $n \geq 1$. On the account of $W_0$ is closed, we obtain $\lim_{n \rightarrow \infty}T^nw\in (w+W_0)$. Hence we conclude that  $\lim_{n \rightarrow \infty}T^nw\in \{w^*\in W: w^*=Tw^*\}\cap (w+W_0)$. 
	\end{proof}	
	
	\section{Applications}
	In this part we provide some applications of our main Theorem \ref{thm3.2}. First we discuss about the occurrence of solution of a fractional differential equation and at the end of this article, we deduce the convergence of iterates for a nonlinear $q$-analogue Bernstein operator. 
	
	A fractional differential equation is a generalization of a differential equation, whose order lies in the set of natural numbers. A fractional derivative looks like $D^\beta$, where $\beta>0$ and the symbol ${}^cD^\beta$ denotes the Caputo fractional derivative \cite{baleanu} with order $\beta$. In fact if $\beta \in \mathbb{N}$, then fractional differential equation turns out to be a normal differential equation. The presence of solution for fractional differential equations are deduced through the perception of fixed point theory, (see \cite{baleanu,khan,gopal}).   
	In this article, we assume a generalized fractional differential equation
	\begin{equation}\label{eqn6}
		{}^cD^\beta w(b)+g(b,f(w(b)))=0~ \textrm{whenever}~~ w\in C[0,1],~ 0<b<1~\textrm{and}~1<\beta,
	\end{equation}
	having boundary criteria $w(0)=0$ and $w(1)=0$, where the maps $f:C[0,1]\to C[0,1]$ and $g:[0,1]\times C[0,1]\to \mathbb{R}$ are both continuous. The Green function $G(b,a)$ corresponding to the equation $(\ref{eqn6})$ takes the value $(b(1-a))^{\beta-1}-(b-a)^{\beta-1}$ whenever $0\leq a\leq b\leq 1$, and it attains the value $\frac{(b(1-a))^{\beta-1}}{\Gamma(\beta)}$ whenever $0\leq b\leq a\leq1$.  
	The collection of every continuous real-valued mappings on $[0,1]$ is represented by $C([0,1])$, which incorporates with $\|.\|=\max_{b\in [0,1]}|w(b)|$. By applying our main Theorem \ref{thm3.2}, in the succeeding result we establish the sufficient criteria for ensuring a solution of the above mentioned fractional differential equation $(\ref{eqn6})$.
	
	\begin{theorem}\label{thm4.1}
		Take the fractional differential equation as stated in $(\ref{eqn6})$. Assume that the succeeding criteria happen:
		\begin{enumerate}
			\item[(i)] for every $v,w\in C[0,1]$ and an element $b$ lies in $[0,1]$,
			\begin{equation*}
				\left|g(b,f(v(b)))-g(b,f(w(b)))\right|\leq k\left(\left\|fv-fw\right\|\right)\left|fv(b)-fw(b)\right|
			\end{equation*} 
		    where $k\in U$,
			\item[(ii)] for each $w\in C[0,1]$, $Fw\subseteq f\left(C[0,1]\right)$ where $F$ is a single valued mapping on $C[0,1]$, which is stated as
			\begin{equation*}
				F(w(b))=\int_0^1G(b,a)g(a,(fw)(a)) da.
			\end{equation*}
		\end{enumerate}
		Then equation $(\ref{eqn6})$ admits a solution if  $ f(C[0,1])$ is closed set in $C[0,1]$.
	\end{theorem}
	\begin{proof}
		We assume that $W=C[0,1]$ incorporate with a graph $G$ having $V(G)=W$ and $E(G)=\{(v,w):(v,w)\in W\times W\}$. Now we visualize that, a function $w$ lies in $W$ meets the equation $(\ref{eqn6})$, if and only if it fulfils the succeeding equation,
		\begin{equation}\label{eqn7}
			w(b)=\int_0^1G(b,a)g(a,(fw)(a))da.
		\end{equation}
		Let $v\in W$ and $fw\in Fv=\int_0^1G(b,a)g(a,(fv)(a))da$. Now
		\begin{eqnarray*}
			|fw(b)-Fw(b)|&= & \left|\int_0^1G(b,a)g(a,(fv)(a))da-\int_0^1G(b,a)g(a,(fw)(a))da\right|\\
			&\leq &\int_0^1\left|G(b,a)\right|\left|g(a,(fv)(a))-g(a,(fw)(a))\right|da\\
			&\leq &k\left(\left\|fv-fw\right\|\right)\int_0^1\left|G(b,a)\right|\left|fv(a)-fw(a)\right|da~\text{[from $(i)$]}\\ 
			&\leq &k\left(\left\|fv-fw\right\|\right)\|fv-fw\|\int_0^1\left|G(b,a)\right|da\\
			&\leq & k\left(\left\|fv-fw\right\|\right)\|fv-fw\|.
		\end{eqnarray*} 
		Therefore $\|fw-Fw\|\leq k(\|fv-fw\|)\|fv-fw\|$. Assume $\{fu_n\}_n \in W$ in order that $fu_n\to fu$. Consequently for each $n$,
		\begin{eqnarray*}
			|fu-Fu|&\leq &|fu-fu_n|+|fu_n-Fu_n|+|Fu_n-Fu|\\
			&\leq & \left|\int_0^1G(b,a)g(a,(fu_n)(a))da-\int_0^1G(b,a)g(a,(fu)(a))da\right|\\  & &~ + |fu-fu_n|+|fu_n-Fu_n|\\
			&\leq & |fu-fu_n|+|fu_n-Fu_n|+k(\|fu_n-fu\|)\|fu_n-fu\|\textrm{[by~ $(i)$]}.
		\end{eqnarray*}
		Therefore for each $n$, $\|fu-Fu\|\leq \|fu-fu_n\|+\|fu_n-Fu_n\|+\|fu_n-fu\|$. Hence $\|fu-Fu\|\leq\liminf_{n\to \infty}\|fu_n-Fu_n\|$. Then using Theorem \ref{thm3.2} we can conclude that $F$ and $f$ has a function $w^*(b)\in W$ in order that $Fw^*(b)=f(w^*(b))$. Hence $f(w^*(b))$ fulfils the equation $(\ref{eqn7})$, therefore the given equation $(\ref{eqn6})$ has a solution.
	\end{proof}
	
	For $n\in\mathbb{N}$ and $q>0$, Lupa\c{s} \cite{lupas} first presented the $q$-analogue Bernstein operator $L_{n,q}$ on the space $C[0,1]$, in the year $1987$. In fact if $q=1$, then this operator reduces to the classical Bernstein operator \cite{KR8}. For $q>0$ and any $\phi$ lies in $C[0,1]$, we consider a nonlinear $q$-analogue Bernstein operator,
	\begin{equation}\label{eq1_Bern}
		\tilde {L}_{n,q}(\phi(a))=\sum_{i=0}^n \left|\phi\bigg(\frac{[i]_q}{[n]_q}\bigg)\right|b_{n,i}(q,a)\quad\textrm{for}~a\in[0,1],
	\end{equation}
	where for each $i\in \mathbb{Z}^{+}$,
	\begin{equation*}
		b_{n,i}(q,a)=\frac{\left[\begin{array}{l}{n}\\{i}\end{array}\right]_q q^{\frac{i(i-1)}{2}}a^i(1-a)^{n-i}}{\prod_{j=0}^{n-1}\{1-a+q^ja\}}, 
		\hspace*{0.15cm} 
		\left[\begin{array}{l}{n}\\{i}\end{array}\right]_q=\frac{[n]_q!}{[n-i]_q![i]_q!},
		\hspace*{0.15cm}
		[n]_q!=\prod_{i=1}^{n}[i]_q,
	\end{equation*}
	and $[0]_q!=1$, where $[i]_q=\sum_{k=0}^{i-1}q^{k}$ for $i\geq 1$ and $[0]_q=0$.
	Agratini \cite{Agra} first studied the convergence of iterates of $q$-analogue Bernstein operator $L_{n,q}$ on the space $C[0,1]$ in the year 2008. In the upcoming corollary we derive the the convergence of iterates for nonlinear  $q$-analogue Bernstein operator $\tilde{L}_{n,q}$ through Theorem \ref{iterates}.
	
	\begin{corollary}
		Suppose that $\tilde{L}_{n,q}~(q>0)$ is an operator as mentioned in $(\ref{eq1_Bern})$. Then for each $\varphi\in C[0,1]$ and fixed $n\in \mathbb{N}$, $\lim_{j\to \infty}\tilde{L}_{n,q}^j\varphi \in \{\varphi^{*}\in C[0,1]:\tilde{L}_{n,q}\varphi^{*}=\varphi^{*}\}$. Furthermore, 
		for each $\varphi\in W^{'}=\{\varsigma\in C[0,1]$: $\varsigma(1)\geq 0$ and $\varsigma(0)\geq 0$\},
		\begin{equation*}
			\lim_{j\to \infty}(\tilde{L}_{n,q}^j\varphi)(a)=\varphi(0)(1-a)+\varphi(1)a \quad \textrm{whereas}~0\leq a\leq 1.
		\end{equation*}
	\end{corollary}
	\begin{proof}
		Consider $W=C[0,1]$ with supremum norm. Let $W_0\subseteq W$ consist of all $\varphi\in W$ whose zeros are $0$ and $1$. Take a function $\phi$ in $W$ in order that $\phi-\tilde{L}_{n,q}\phi\in W_0$. Then for each $a$ lies in $[0,1]$ we have,
		\begin{eqnarray}
			\big|\tilde{L}_{n,q}\phi(a)-\tilde{L}_{n,q}^2\phi(a)\big|&\leq&\sum_{i=1}^{n-1} b_{n,i}(q,a) \left|(\phi-\tilde{L}_{n,q}\phi)\left(\frac{[i]_q}{[n]_q}\right)\right|\nonumber\\
			&\leq&\left[1-b_{n,n}(q,a)-b_{n,0}(q,a)\right]\|\phi-\tilde{L}_{n,q}\phi\|.\label{eq10_application2}
		\end{eqnarray}
		Consider $h_{n}(a)=q^{n(n-1)/2}a^n+(1-a)^n$, then for $n=1$, the minimum value of $h_{n}(a)$ is $1$. For $n \geq 2$, the minimum value of $h_{n}(a)$  is $\left(q^{n/2}/(1+q^{n/2})\right)^{n-1}$, which occurs at $a=(1+q^{n/2})^{-1}$. 
		On the other hand consider $g_{n}(a)=\prod_{s=0}^{n-1}(1-a+q^s a)$, then one can easily verify that maximum of $g_{n}(a)$ over $[0,1]$ is less than or  equal to $\max\{1,q^{(n-1)^2}\}$. Consequently, 
		\begin{equation}
			\frac{1}{\max\{q^{(n-1)^2},1\}}\bigg(\frac{q^{n/2}}{1+q^{n/2}}\bigg)^{n-1} \leq \displaystyle{\frac{h_{n}(a)}{g_{n}(a)}}= b_{n,n}(q,a)+b_{n,0}(q,a).
			\label{eq9_application2}
		\end{equation}
		Again we observe that $0<b_{n,q}=\frac{1}{\max\{1,q^{(n-1)^2}\}}\bigg(\frac{q^{n/2}}{1+q^{n/2}}\bigg)^{n-1} \leq 1$. Subsequently from (\ref{eq10_application2}) and (\ref{eq9_application2}), it follows that
		$||\tilde{L}_{n,q} \phi-\tilde{L}_{n,q}^2\phi||\leq k(||\phi-\tilde{L}_{n,q} \phi||)||\phi-\tilde{L}_{n,q} \phi||$, where $\forall\,t \geq 0$, $k(t)=(1-b_{n,q})\in U$.
		Consider a sequence $\{\eta_j\}_j\in W$ in order that $\eta_j\to \eta$. Then for every $a\in[0,1] $ and $j\in \mathbb{Z}^{+}$,
		\begin{equation*}
			\begin{aligned}
			|\eta(a) -\tilde{L}_{n,q}\eta(a)|&\leq|\eta(a)-\eta_j(a)|+|\eta_j(a) -\tilde{L}_{n,q}\eta_j(a)|+|\tilde{L}_{n,q}\eta_j(a)-\tilde{L}_{n,q}\eta(a)|\\
			&\leq \|\eta-\eta_j\|+\|\eta_j -\tilde{L}_{n,q}\eta_j\|+\|\eta_j-\eta\|.
			\end{aligned}
		\end{equation*}
		Therefore $\|\eta -\tilde{L}_{n,q}\eta\|\leq\|\eta-\eta_j\|+\|\eta_j -\tilde{L}_{n,q}\eta_j\|+\|\eta_j-\eta\|$. Now taking $j\to \infty$, we have $\|\eta -\tilde{L}_{n,q}\eta\|\leq\liminf_{j\to \infty}\|\eta_j -\tilde{L}_{n,q}\eta_j\|$. Again it is simple to observe that $(I-\tilde{L}_{n,q})(W)\subseteq W_0$. Then using the Theorem $\ref{iterates}$, for any $\varphi\in W$, 
		$\displaystyle{\left\{\lim_{j\to \infty}\tilde{L}_{n,q}^j\varphi\right\}}\subseteq(\varphi+W_0)\cap Fix\,\tilde{L}_{n,q}$. It is simple to verify that the set $(\varphi+W_0)\cap Fix\,\tilde{L}_{n,q}$ is singleton, also $a$ and $1-a$ are fixed points of $\tilde{L}_{n,q}$. Again for each $\varphi\in W^{'}$, $\varphi(0)(1-a)+\varphi(1)a \in Fix\,\tilde{L}_{n,q} $. Furthermore, $\varphi(0)(1-a)+\varphi(1)a\in \varphi+W_0$. Consequently, we conclude that 	$\lim_{j\to \infty}(\tilde{L}_{n,q}^j\varphi)(a)=\varphi(0)(1-a)+\varphi(1)a$ for each $\varphi\in W^{'}$.
	\end{proof}

	
	%
	%
	%
	%
	%
	%
	%
	%
	%
	%
	%
	
	\section*{Acknowledgment}
	The first author is thankful to the Ministry of Human Resource Development (MHRD), India for the financial support.


\end{document}